\documentclass[11pt]{amsart}
\usepackage{amsfonts,amsthm,amssymb,mathrsfs,amsmath, url, xspace}
\usepackage[utf8]{inputenc}
\usepackage[T1]{fontenc}

\newtheorem{lem}{Lemma}
\newtheorem{lemma }{Lemma}

\newtheorem{theorem}{Theorem}
\newtheorem {conj }{Conjecture}

\newtheorem{remark }{Remark}
\newtheorem {assumption }{Assumption}

\newtheorem*{rec-hyp}{Recursion Hypothesis}

\newcommand{\Ocal}{\mathcal{O}}

\newcommand{\ve}{\varepsilon}
\newcommand{\Fainleib}{Fa\u{\i}nle\u{\i}b\xspace}

\DeclareMathOperator{\1}{1\!\!\!1}

\title{A Higher Order Levin-Fa\u{\i}nle\u{\i}b Theorem}
\author{Olivier Ramar\'{e}}
\address{
  CNRS / Institut de Math\'ematiques de Marseille\\
Aix Marseille Universit\'e, U.M.R. 7373\\
Site Sud, Campus de Luminy, Case 907\\
13288 MARSEILLE Cedex 9, France}
\email{olivier.ramare@univ-amu.fr}
\author{Alisa Sedunova}
\address{
  Chebyshev Lab,
  14th Line 29B, Vasilyevsky Island,
  St. Petersburg, 199178, Russia}
\email{alisa.sedunova@phystech.edu}

\author{Ritika Sharma}
\address{Stockholm University, Frescativägen, 11419, Stockholm
}
\email{rish0842@student.su.se}
  
\begin{document}

\begin{abstract}
  When restricted to some non-negative multiplicative function, say~$f$,
  boun\-ded on primes and that vanishes on non square-free integers, our result
  provides us with an asymptotic for $\sum_{n\le X}f(n)/n$ with error term
  $\Ocal((\log X)^{\kappa-h-1+\ve})$ (for any positive $\ve>0$) as soon
  as we have
  $\sum_{p\le Q}f(p)(\log p)/p=\kappa\log Q+\eta+\Ocal(1/(\log
  2Q)^h)$ for a non-negative~$\kappa$ and some non-negative integer~$h$. The method generalizes the 1967-approach of Levin and
  \Fainleib and uses a differential equation.
\end{abstract}


\maketitle

\section{Introduction}

In 1908, E. Landau \cite{Landau*08-2}, was the first to obtain an
asymptotic formula for the number of integers up to a given number
that are sum of two coprime squares.  He used analytical method, which
involves considering the squareroot of some analytical function and
avoiding its pole through Hankel contour. Later, this procedure was
further developed by H. Delange and A. Selberg allowing them to obtain
asymptotic for partial sums of arithmetic functions whose Dirichlet
series can be written in terms of complex powers of the Riemann
$\zeta$-function. This is now often referred to as the Selberg-Delange
method. In \cite{MR0229600}, B.V. {Levin} and A.S. {\Fainleib}
established the logarithmic density of the same set by an elementary
argument under more general conditions. When combined with the earlier
method of E. Wirsing \cite{Wirsing*61}, as was done in
\cite{Ramare*06}, this leads to the determination of the natural
density as well.

\noindent In \cite{Serre*76}, J.-P. Serre used Landau's method to
examine several other cases and deployed it to encompass not only the
main term but also an asymptotic development, leading to a better
error term. Extending the Levin and \Fainleib approach in a similar
fashion would allow for more general hypotheses as well. This is the
aim of the present paper.  To express our results, we take a
non-negative multiplicative function~$f$ and, following Levin
and \Fainleib, we associate to it the function $\Lambda_f(n)$ which
is~0 when $n$ is not a prime power and which is otherwise defined by
the formal power expansion:
\begin{equation}
    \sum_{k\ge0}\frac{\Lambda_f(p^k)}{p^{ks}}
    =\biggl(\sum_{k\ge1}\frac{f(p^k)\log p}{p^{ks}}\biggr)/ \sum_{k\ge0}\frac{f(p^k)}{p^{ks}}.
\end{equation}
We recall some of its properties in Section~\ref{Lf}. To handle the uniformity in our result, we recall that we use $f=\Ocal^*(g)$ to mean that $|f|\le g$ and $f=\Ocal_{A,h,\kappa}(g)$ to mean that $|f|\le C(A,\kappa,h)g$, where the constant $C(A,h,\kappa)$ depends only on the stated parameters. 
Here is our main theorem.
\begin{theorem}
\label{thmmain}
Let $f$ be a non-negative multiplicative function. Assume that, for
some integer $h\ge0$, one has
  \begin{equation}
  \label{Hh}\tag{$\text{H}_h$}
  \forall Q\ge1,\quad
      \sum_{m\le Q}\frac{\Lambda_f(m)}{m}
      =\kappa \log Q+\eta_0+\Ocal^*(A/\log^h(2Q))
  \end{equation}
  for some constants $\kappa\ge0$, $A$ and $\eta_0$. We further assume
  that $|\eta_0|\le A$. Then there exist constants $C$ and
  $(a_k)_{1\le k\le h}$ such that, when $X\ge 3$, we have
  \begin{align*}
      \sum_{n\le X}\frac{f(n)}{n}(\log n)^{h+1}
      =
      C(\log X)^{\kappa+h+1}\biggl(
      1+\frac{a_1}{\log X}+\ldots+
      \frac{a_{h}}{(\log X)^{h}}\biggr)
      \\+\Ocal_{A,\kappa,h}\bigl((\log X)^{\kappa}(\log\log (3X))^{\frac{(h+2)(h+1)}{2}}\bigr),
  \end{align*}
  where
  \begin{equation*}
    C=\frac{1}{\Gamma(\kappa+1)} \prod_{p\geq 2}\biggl(
    \biggl(1-\frac{1}{p}\biggr)^\kappa
    \sum_{\nu \ge 0} \frac{g(p^\nu)}{p^\nu}\biggr).
\end{equation*}
\end{theorem}
We have a same error term for the sum
\begin{equation*}
  \sum_{n\le X}\frac{f(n)}{n}\biggl(\log \frac{X}{n}\biggr)^{h+1}.
\end{equation*}
We can also obtain $\sum_{n\le X}\frac{f(n)}{n}(\log n)^{k}$ for any
$k\in\{0,\ldots,h\}$ with an error term
$\Ocal((\log\log X)^{\frac{(h+2)(h+1)}{2}}/(\log X)^{h+1-k})$, by summation by parts, but some
additional $\log\log X$ term may appear in the development when $\kappa$ is an
integer, which is why we state our result in this manner.
The non-negative assumption is not essential in our
method, nor is the fact that $f$ is real valued (but $\kappa$ has to
be a real number), we may instead assume that
\begin{equation}
  \label{eq:13}
  \sum_{n\le X}|f(n)|\ll (\log X)^{\kappa^*}
\end{equation}
for some parameter $\kappa^*$ and modify our error term
$\Ocal((\log X)^\kappa(\log\log X)^c)$ to
$\Ocal((\log X)^{\kappa^*}(\log\log X)^c)$. This is for instance the path
chosen, when $h=0$ in Theorem 1.1 of the book
\cite{Iwaniec-Kowalski*04} by H.~Iwaniec and E.~Kowalski.
We did not try to optimize the power of $\log\log (3X)$ that appears. It
is likely that no such term should be present in fact, but in
practice, when our assumption holds for $h\ge1$, it holds for any
$h$. Using the result for $h+1$ removes this parasitic factor.

To measure the relative strength our theorem, let us mention that, with $h=1$
for instance and $\mu$ is the Moebius function, it gives us
a proof that the estimate
$\sum_{p\le X}(\log p)/p=\log X+c+\Ocal(1/\log X)$ implies that
$\sum_{n\le X}\mu(n)/n\ll (\log\log X)^3/\log X$. The case $h>1$ yields
another proof of the results of A.~Kienast in~\cite{Kienast*25}.

A.~Granville and D.~Koukoulopoulos considered a similar question
in~\cite{Granville-Koukoulopoulos*19}, our hypotheses are in some
place weaker, as we consider averages of $f(p)/p$ rather than averages
of $f(p)$ and no boundedness condition on $f(p)$ is asked for, but we
require that $f$ is non-negative. However the main difference truly
comes at the methodological level: our proof stays in the realm of
real analysis while Granville and Koukoulopoulos use complex analysis
around the Perron summation formula. The readers may also consult
\cite[pp. 183--185]{Selberg*91b} by A.~Selberg,
\cite{Murty-Saradha*93} by M.R.~Murty and N.~Saradha, and
\cite{Moree-Riele*04} by P.~Moree and H.J.J.~te Riele on related issues.

The proof relies on a recursion on $h$. It is however easier to
assume a more complete hypothesis.
\begin{rec-hyp}[for $h$]
  For each $\ell\in[0, h+1]$, there exists a polynomial $P_{\ell}$ of
  degree $\ell$ such that
  \begin{equation}
      \sum_{n\le X}\frac{f(n)}{n}(\log n)^\ell
      =
      \biggl(P_{\ell}(\log X)
      +\Ocal\bigl((\log \log
      X)^{\frac{(h+1)(h+2)}{2}}\bigr)\biggr)
      (\log X)^\kappa.
  \end{equation}
\end{rec-hyp}
We show during the proof that we may as well assume a similar
hypothesis with $(\log(X/n))^\ell$ rather than $(\log n)^\ell$: this
is a consequence of the functional relation we prove at the
beginning of our proof, see~\eqref{step1}.
The Levin-\Fainleib Theorem gives a proof of this claim
when~$h=0$ (and even better as the $\log\log(3X)$ is absent in this theorem).
We provide in Section~\ref{Technical} a survey of the proof.

\subsection*{Notation}

We set for typographical simplicity $g(n)=f(n)/n$.
Next, for a non-negative integer $j$ define
\begin{equation} \label{Gdef}
    G_j(X)=\sum_{n \le X}g(n)\log^j (X/n), \quad G_0(X)=G(X).
\end{equation}

For $k \ge 0$, we define $H_k(\log X)=G_k(X)$.

\subsection*{Acknowledgments}
This paper started in 2018 when the first and third authors were
invited by the Indian Statistical Institute of Delhi under Cefipra
program 5401-A. It was continued when these authors were visiting
Stockholm in early 2019 and then in July of the same year when both
first and second authors were invited by the Max Planck Institute in
Bonn. It was finalized in 2021 when the first author was invited by
the Haussdorf Institut f\"ur Mathematik in Bonn and the second one was
invited by the Max Planck Institute in Bonn. These bodies are to be
thanked warmly for providing suitable conditions without which this
piece of work would surely have died in our drawers.

\section{On the function $\Lambda_f$}
\label{Lf}
Let $F$ denotes the formal Dirichlet series of $f$, namely
\begin{equation*}
    F(s)=\sum_{n\ge1}\frac{f(n)}{n^s}.
\end{equation*}
Note that Euler product formula gives
\begin{equation*}
    F(s)=\prod_{p\ge2}\biggl(1+\sum_{k\ge1}\frac{f(p^k)}{p^{ks}}\biggr).
\end{equation*}
On taking the logarithmic derivative of $F(s)$, we find that
\begin{equation*}
  -\frac{F'(s)}{F(s)} =
  \sum_{p \ge 2} \biggr(\sum_{k\ge1}\frac{ f(p^k)}{p^{ks}} \log
  (p^k)\biggr)
  \biggr(1+\sum_{k\ge1}\frac{f(p^k)}{p^{ks}}\biggr)^{-1}
  = \sum_{p \ge 2} Z_p(s) \log p.
\end{equation*}
Further, expanding the second product in $Z_p(s)$
and changing the order of summation we find that
\begin{align*}
  Z_p(s) &= \sum_{k\ge1}\frac{kf(p^k)}{p^{ks}}\sum_{r\ge0}(-1)^r
           \sum_{\ell\ge0}\sum_{k_1+k_2+\ldots+k_r=\ell}
           \frac{f(p^{k_1})\cdots f(p^{k_r})}{p^{\ell s}} \\
         &=
           \sum_{m\ge1}\frac{1}{p^{ms}}\biggl(
           \sum_{k+k_1+\ldots +k_r=m}
           (-1)^r kf(p^k)f(p^{k_1})\cdots f(p^{k_r}) \biggl).
\end{align*}
Thus,
\begin{equation} \label{Flogder}
   -\frac{F'(s)}{F(s)} = \sum_{n\ge1}\frac{\Lambda_f(n)}{n^s},
\end{equation}
where
\begin{equation}
    \Lambda_f(p^m)=\sum_{k+k_1+\ldots +k_r=m}
    (-1)^r kf(p^k)f(p^{k_1})\cdots f(p^{k_r})
    \log p
\end{equation}
and $\Lambda_f(n)=0$ when $n$ is not a prime power.  Note that
$\Lambda_f(p^m)$ depends only on the local factor of $F(s)$ at prime
$p$. In particular $\Lambda_1(p^m)=\Lambda(p^m)$.  Moreover,
when $f(p^m)=1_{p\in \mathcal{P}}$, we have
$\Lambda_f(p^m)=\Lambda(p^m) \cdot f(p^m)$ (here $1_X=1$ if $X$ is
true and $0$ otherwise).  For example, let us select
$\mathcal{P} = \{ p \equiv 1 \pmod
4\}$.  As mentioned above, the definition of $\Lambda_f(p^m)$ depends
only on the local factor at prime $p$, hence we readily see that
$\Lambda_f(p^m)=\Lambda(p^m)$ for $p \equiv 1 \pmod 4$ and $0$
otherwise.  Note that when $f$ is supported on square-free integers, we
get $\Lambda_f(p^m)=(-1)^{m-1}f(p)^m\log p$.

\begin{lem}
\label{errorterm}
Let $k$, $h$ be two non-negative real numbers. Then for any $k\leq h$, there exists a constant $\eta_k$, such that, under
assumption \eqref{Hh} we have
  \begin{equation}\tag{A} \label{1}
    \sum_{n\le Q}\frac{\Lambda_f(n)\log^k n}{n}=\frac{ \kappa}{k+1}\log^{k+1}Q +   \eta_k+
    E_{k,h}(Q),
\end{equation}
where $E_{k,h}(Q) \ll 1/\log^{h-k}(2Q)$ for $k < h$ and  $E_{h,h}(Q) \ll \log \log (3Q)$.
\end{lem}

\begin{proof}
  Denote the sum on the left hand side of \eqref{1} by $S_k(Q)$. Then
  using partial summation, we have
  \begin{equation*}
      S_k(Q) = S_0(Q) \log^k Q-k\int_1^Q S_0(t)\log^{k-1}t \, \frac{dt}{t}.
  \end{equation*}
  Further, when $k<h$, we may apply \eqref{Hh} to get
  \begin{align*}
      S_k(Q)
      = \frac{ \kappa}{k+1}\log^{k+1}Q &+  \eta_0 \log^kQ
      -  \eta_0 k \int_1^Q \frac{\log^{k-1}t dt}{t}\\
      &-k\int_1^\infty \Bigl(S_0(t)- \kappa\log t-  \eta_0\Bigr)
     \frac{\log^{k-1} t dt}{t}\\
      &+\Ocal\biggl(\frac{1}{\log^{h-k}Q}+
      \int_Q^\infty \frac{d\log t}{\log^{h-k+1} t}
      \biggr),
  \end{align*}
  whence
  \begin{equation*}
      S_k(Q)
      =
     \frac{ \kappa}{k+1}\log^{k+1} Q+  \eta_k+\Ocal(1/\log^{h-k} (2Q))
  \end{equation*}
  as announced. Analogous argument gives the result for $k=h$.
\end{proof}

\section{Generalizations of $\Lambda_f$}
\label{Lfh}

We will use the next formula several times.
\begin{lem}[Fa\`a di Bruno Formula]
  We have
  \begin{equation*}
    \frac{d^nf(g(x))}{dx^n}
    = \mkern-40mu\sum_{\substack{m_1,m_2,\cdots,m_n\ge0,\\
        m_1+2m_2+\cdots+nm_n=n}}
    \mkern-20mu\frac{n!}{m_1!m_2!\cdots m_n!}f^{(m_1+m_2+\cdots+m_n)}(x)
    \prod_{j=1}^n\biggl(\frac{g^{(j)}(x)}{j!}\biggr)^{m_j}.
  \end{equation*}
\end{lem}

Here is a combinatorial identity, which is an
immediate corollary of \cite[Theorem 2.1]{Ramare*12-2}, itself being a
straightforward consequence of the Fa\`a di Bruno Formula.
\begin{lem}
  \label{ident}
  Let $F$ be a function and denote $Z_F=-F'/F$.
 We have
  \begin{equation*}
    F^{(h+1)}=
    F\sum_{\sum_{i\ge1}ik_i=h+1}
    \frac{(h+1)!(-1)^{\sum_i k_i}}{k_1!k_2!\cdots
      (1!)^{k_1}(2!)^{k_2}\cdots}
    \prod_{k_i}Z_F^{(i-1)k_i}.
  \end{equation*}
  Notation $Z_F^{(i-1)k_i}$ denotes the $(i-1)$-th derivative
  multiplied $k_i$ times by itself.
\end{lem}
\begin{proof}
  This is an immediate corollary of \cite[Theorem 2.1]{Ramare*12-2}
  with $F=1/G$ and hence $Z_F=-Z_G$.
\end{proof}
When $h=1$, this gives $F''=F(Z_F^2-Z_F')$.
We thus define
 \begin{equation}
 \label{defLambdafh}
    \sum_{n\ge1}\frac{\Lambda_{f,h}(n)}{n^s}
    =(-1)^{h}\sum_{\sum_{i\ge1}ik_i=h}
    \frac{h!(-1)^{\sum_i k_i}}{k_1!k_2!\cdots
      (1!)^{k_1}(2!)^{k_2}\cdots}
    \prod_{k_i}Z_F^{(i-1)k_i}
  \end{equation}
so that
\begin{equation*}
  f\log^h=f\star\Lambda_{f,h}.
\end{equation*}
When $f=\1$, these functions have their origin in the work of
A.~Selberg~\cite{Selberg*49b} around an elementary proof of the Prime
Number Theorem. They have been generalized as above by E.~Bombieri in
\cite{Bombieri*76}, see also the papers~\cite{Friedlander-Iwaniec*78}
and~\cite{Friedlander-Iwaniec*96-1} by J.~Friedlander and
H.~Iwaniec. Incidentally, Lemma~\ref{ident} gives a non-recursive
description of the functions $\Lambda_h=\Lambda_{\1,h}$, something
that is missing from the aforementioned works.

\begin{lem}
  Let $\theta_1$ and $\theta_2$ be two functions on the integers that
  satisfy, for $i\in\{1,2\}$,
  \begin{equation*}
    \sum_{n\le X}\theta_i(n)=C_i(\log X)^{d_i}+Q_i(\log X)
    +\Ocal(1/(\log 2X)^{h-d_i})
  \end{equation*}
  where $d_i\ge1$, $Q_i$ is a polynomial of degree at most $d_i-1$ and
  $h$ is some fixed parameter. Then
  \begin{multline*}
    \sum_{mn\le X}\theta_1(m)\theta_2(n)
    =C_1C_2\frac{d_1!d_2!}{(d_1+d_2)!}(\log X)^{d_1+d_2}
    +Q(\log X) \\+\Ocal\biggl(\frac{1}{(\log 2X)^{h-d_1-d_2}}\biggr),
  \end{multline*}
  where $Q$ is a polynomial of degree at most $d_1+d_2-1$.
\end{lem}
\begin{proof}
  We use the Dirichlet Hyperbola Formula.  We split the variables at
  $\sqrt{X}$ to get the announced error term. In order to compute the
  main term, it is enough to consider
  \begin{equation*}
      S=\sum_{n\le X}\theta_1(n)
      C_2\biggl(\log\frac{X}{n}\biggr)^{d_2}.
  \end{equation*}
  An integration by parts gives us
  \begin{align*}
      S 
      &= 
      C_2\sum_{n\le X}\theta_1(n)
      d_2\int_{1}^{X/n}(\log t)^{d_2-1}\frac{dt}{t}
      \\&=
      C_2d_2\int_{1}^{X}
      \sum_{n\le X/t}\theta_1(n)
      (\log t)^{d_2-1}\frac{dt}{t},
  \end{align*}
  so that the principal part of the main term is given by
  \begin{align*}
      M&=
      C_1C_2d_2\int_{1}^{X}
      \biggl(\log\frac{X}{t}\biggr)^{d_1}
      (\log t)^{d_2-1}\frac{dt}{t}
      \\&=
      C_1C_2d_2(\log X)^{d_1+d_2}
      \int_0^1(1-u)^{d_1}u^{d_2-1}du
      \\&=C_1C_2\frac{d_1!d_2!}{(d_1+d_2)!}(\log X)^{d_1+d_2}
  \end{align*}
  by the classical evaluation of the Euler beta-function.
\end{proof}
On iterating the previous lemma, we
get the next one.
\begin{lem}
  \label{prep}
  Let $(\theta_i)_{i\le r}$ be $r$ functions on the integers that
  satisfy, for $i\in\{1,\cdots,r\}$,
  \begin{equation*}
    \sum_{n\le X}\theta_i(n)=C_i(\log X)^{d_i}+Q_i(\log X)
    +\Ocal(1/(\log 2X)^{h-d_i}),
  \end{equation*}
  where $d_i\ge1$, $Q_i$ is a polynomial of degree at most $d_i-1$ and
  $h$ is some fixed parameter. Then
  \begin{multline*}
    \sum_{m_1\cdots m_r\le X}\prod_{i\le r}\theta_i(m_i)
    =\prod_{i\le r}C_i\frac{d_1!\cdots d_r!}{(d_1+\cdots+d_r)!}
    (\log X)^{d_1+\cdots+d_r}
    +Q(\log X) \\+\Ocal\biggl(\frac{1}{(\log 2X)^{h-d_1-\cdots-d_r}}\biggr),
  \end{multline*}
  where $Q$ is a polynomial of degree at most $d_1+d_2+\cdots+d_r-1$.
\end{lem}
\begin{lem}
  \label{lifeagain}
  Under~\eqref{Hh}, we have
  \begin{multline*}
    \sum_{n\le X}
   \frac{\Lambda_{f,k}(n)}{n}
    =
    \frac{\kappa(\kappa+1)\cdots (\kappa+k-1)}{k!}
    (\log X)^k
    \\+Q(\log X)+\Ocal\biggl(
    \frac{\log\log(3X)}{(\log 2X)^{h+1-k}}
    \biggr).
  \end{multline*}
  where $Q$ is polynomial of degree at most $k-1$.
\end{lem}

\begin{proof}
  Lemma~\ref{prep} tells us that the sum reads
  \begin{multline}
    \sum_{n\le X}
     \frac{\Lambda_{f,k}(n)}{n}
    =\sum_{\sum_{i\ge1}ik_i=k}\frac{k!}{k_1!k_2!\cdots
      (1!)^{k_1}(2!)^{k_2}\cdots} \prod_i\frac{\kappa^{k_i}i!^{k_i}}{i^{k_i}}
    \frac{(\log X)^{k}}{k!}
    \\+Q(\log X)+\Ocal\biggl(
    \frac{\log\log(3X)}{(\log 2X)^{h+1-k}}
    \biggr).
  \end{multline}
  The main term simplifies into
  \begin{equation*}
    \sum_{\sum_{i\ge1}ik_i=k}\frac{1}{k_1!k_2!\cdots} \prod_i\frac{\kappa^{k_i}}{i^{k_i}}
    (\log X)^k.
  \end{equation*}
  The $i$-th derivative of $g(x)=-\kappa\log (1-x)$ is
  $(i-1)!\kappa/(1-x)^i$
  so that $\kappa/i $ is also $g^{(i)}(0)/i!$. The Fa\`a di Bruno
  Formula for the $k$-th derivative of $\exp(g(x))=(1-x)^{-\kappa}$ tells us that
  \begin{equation*}
    \sum_{\sum_{i\ge1}ik_i=k}\frac{k!}{k_1!k_2!\cdots}
    \prod_i\frac{\kappa^{k_i}}{(i(1-x)^{i})^{k_i}}
    =\frac{\kappa(\kappa+1)\cdots (\kappa+k-1)}{(1-x)^{\kappa+k}}.
  \end{equation*}
  We evaluate this equality at $x=0$.
\end{proof}
\section{Auxiliary results}

\begin{lem} \label{giter}
For $k \ge 1$ we have $\displaystyle G_k(X)=k\int_1^XG_{k-1}(t) \frac{dt}{t}$. 
\end{lem}

\begin{proof}
Notice that by a simple change of variable $t=\log(u/n)$ we have
\begin{equation*}
  \frac{1}{k} \biggl(\log\frac{X}{n}\biggr)^{k} 
  = \int_0^{\log (X/n)} t^{k-1}dt 
  = \int_{n}^{X} \biggl(\log\frac{u}{n}\biggr)^{k-1} \frac{du}{u}. 
\end{equation*}
Using the above together with the definition of $G_k$ we directly compute
\begin{equation*}
\begin{split}
    \int_1^X G_{k-1}(t) \frac{dt}{t}
    &=
    \sum_{n\le X}g(n)\int_{n}^{X}\biggl(\log \frac{t}{n}\biggr)^{k-1}\frac{dt}{t}
    \\
    &=\frac{1}{k} \sum_{n\le X}g(n)\biggl(\log\frac{X}{n}\biggr)^{k} = \frac{G_k(X)}{k}
\end{split}
\end{equation*}
as claimed in the lemma.
\end{proof}
Here is a direct consequence of the previous lemma, on recalling that $H_k(\log X)=G_k(X)$.
\begin{lem}
  \label{hiter}
  When $\ell\in\{0,\ldots,\ell\}$, we have
  $\displaystyle
      H_k^{(\ell)}(u)
      =\frac{k!}{(k-\ell)!}G_{k-\ell}(e^u)$.
\end{lem}
\begin{lem}
  \label{reph}
  When $k\ge0$, we have
  $\displaystyle
      \sum_{n\le e^u}g(n)(\log n)^k
      =\frac{u^{k+1}}{k!}(H_k(u)/u)^{(k)}$.
\end{lem}
\begin{proof}
  This lemma is true for $k=0$. For $k=1$, we find that
  \begin{equation*}
      u^{2}(H_1(u)/u)^{(1)}
      =uH'_1(u)-H_1(u)
      =\sum_{n\le e^u}g(n)
      \big(u-(u-\log n)\bigr)
  \end{equation*}
  as required. For general $k$, write
   \begin{align*}
      \sum_{n\le e^u}g(n)(\log n)^k
      &=\sum_{n\le e^u}g(n)\Bigl(u-\log\frac{e^u}{n}\Bigr)^k
      \\&=\sum_{0\le j\le k}\binom{k}{j}
      u^{j}(-1)^{k-j} G_{k-j}(e^u)
      \\&=\sum_{0\le j\le k}\binom{k}{j}
      u^{j}(-1)^{k-j} \frac{(k-j)!}{k!}H_{k}^{(j)}(u).
  \end{align*}
  We next notice that
  \begin{equation*}
      \frac{d^\ell}{du^\ell}\frac{1}{u}
      =\frac{(-1)^\ell\ell!}{u^{\ell+1}}
  \end{equation*}
  so that
  \begin{align*}
      \sum_{n\le e^u}g(n)(\log n)^k
      &=\frac{u^{k+1}}{k!}\sum_{0\le j\le k}\binom{k}{j}
      (-1)^{k-j} \frac{(k-j)!}{u^{k-j+1}}H_{k}^{(j)}(u)
      \\&=
      \frac{u^{k+1}}{k!}(H_k(u)/u)^{(k)}
  \end{align*}
  as announced.
\end{proof}

\section{Approximate solutions of an Euler differential equation}
In \cite{Popa-Pugna*16} and building on D.~Popa and G.~Rasa
\cite{Popa-Rasa*11},
Popa and Pugna studied perturbation of an Euler
differential equation, say
\begin{equation}
\label{refeq}
    u^ry^{(r)}(u)+\sum_{0\le i\le r-1}b_i u^i y^{(i)}(u)
\end{equation}
for a function $y$ that is in $C^{r}(I)$ for some interval
$I\subset[0,\infty)$. On looking more closely at their work which goes
by iteration, one sees that the last derivative does not need to be
continuous provided one may integrate, and so may be simply
\emph{absolutely continuous on every subinterval of $I$}. We denote
this class by $C^{r-}(I)$.

We next need a second modification of their work. For any $c\in I$,
any complex number $\alpha$ and any suitable function $\varphi$, they
consider
\begin{equation*}
    \Phi^*_{\alpha, c}(\varphi)(x)
    =x^{\Re \alpha}\biggl|\int_c^x
    u^{-\Re \alpha}\varphi(u)\frac{du}{u}\biggr|.
\end{equation*}
Please notice that Popa and Pugna forgot this change of variable that
is necessary between their Theorems 2.1 and 2.3. This explains our
notation $\Phi^*$ rather than the $\Phi$ that these two authors have.
We have added the index $c$ to their notation and we may in fact take
$c=\infty$ (and reverse the order of integration as usual). We select
$r$ parameters $c_1,\ldots,c_r$, some of them maybe be infinite.

Following Popa and Pugna, we consider the root
$\lambda_1,\ldots,\lambda_r$ of the equation
\begin{equation}
    \label{chareq}
     b_0+\sum_{1\le s\le r}\lambda(\lambda-1)\cdots (\lambda-s+1)b_{s}=0.
\end{equation}
We also select a function $S$ in $C^r(I)$.
With these notations, here is the version
of \cite[Theorem 2.3]{Popa-Pugna*16} that we shall use.
\begin{lem}
\label{popa}
Let $\varphi: I\rightarrow [0,\infty)$ be such that
$\Phi^*_{\lambda_r,c_r}\circ\cdots\circ
\Phi^*_{\lambda_1,c_1}(\varphi)$ exists and is finite. Then for every
$y\in C^{r-}(I)$ satisfying
\begin{equation*}
    \forall u\in I,\quad \biggl|
    u^ry^{(r)}(u)+\sum_{0\le i\le r-1}b_i u^i y^{(i)}(u)-S(u)
    \biggr|\le \varphi(u)
\end{equation*}
there exists a solution $y_0$ of
\begin{equation*}
    u^ry^{(r)}(u)+\sum_{0\le i\le r-1}b_i u^i y^{(i)}(u)=0
\end{equation*}
with the property
\begin{equation*}
  \forall u\in I,\quad |y(u)-y_0(u)|\le
  \Phi^*_{\lambda_r,c_r}\circ\cdots\circ \Phi^*_{\lambda_1,c_1}(\varphi)(u).
\end{equation*}
\end{lem}

\section{A differential equation}

On using Lemma \ref{ident} and \ref{lifeagain}, we get
\begin{align} \label{step1}
  \sum_{n\le X}g(n)(\log n)^{h+1}
  =\sum_{n\le X}g(n)
  \biggl(
   &\frac{\kappa(\kappa+1)\cdots (\kappa+h)}{(h+1)!}
    \biggl(\log \frac{X}n\biggr)^{h+1}
   \\& \notag+Q(\log (X/n))+\Ocal(
    \log\log (3X)
    )
  \biggr).
\end{align}
hence, by our recursion hypothesis in $h$, we get
\begin{multline}
  \label{step2}
  \sum_{n\le X}g(n)(\log n)^{h+1}
  =\frac{\kappa(\kappa+1)\cdots (\kappa+h)}{(h+1)!}G_{h+1}(X)
  \\+P(\log X)(\log X)^\kappa+\Ocal(
    (\log X)^\kappa (\log\log X)^{\frac{(h-1)h}{2}}
    )
  \end{multline}
  for some polynomial $P$ of degree at most $h$. Here we have used the
  recursion hypothesis with a $(\log X/n)^k$. It is precisely
  Equation~\eqref{step2} that allows us to switch easily from one form
  of our
  hypothesis to the other. When $h=1$, so $h-1=0$, we do not have
  power of $\log\log X$.

  We may express the
  left-hand side by Lemma~\ref{reph}, getting our first fundamental
  formula:
\begin{multline}
  \label{step0}
  u^{h+1}\biggl(\frac{H_{h+1}(u)}{u}
  \biggr)^{(h+1)}
  =\frac{(\kappa+h)!}{(\kappa-1)!}\frac{H_{h+1}(u)}{u}
\\+(h+1)!P(u)u^{\kappa-1}+\Ocal(
    u^{\kappa-1}(\log u)^{\frac{h(h-1)}{2}}
    ),
  \end{multline}
where we use the shortcut
\begin{equation*}
    \frac{(\kappa+h)!}{(\kappa+h-j)!}
    =(\kappa+h)\cdots(\kappa+h-j+1).
\end{equation*}
This is an \emph{Euler} differential equation. As mentioned before, it may be reduced to a linear differential equation with constant coefficients with the change of variables $u=e^v$, but we shall skip this step and use an already made result. It is technically clearer to first extract a 'simplifying term' and this is our first step.
\subsection*{Simplifying the equation}
Since we may assume that the polynomial $P$ has no constant coefficient, we set
\begin{equation*}
    (h+1)!P(u)=\sum_{1\le s\le h}q_su^s.
\end{equation*}
We define, for $0\le s\le h-1$, the real number $a_s$ by
\begin{equation*}
    \biggl(\frac{(\kappa+s)!}{(\kappa-1)!}-\frac{(\kappa+h)!}{(\kappa-1)!}\biggr)a_s=q_{s-1}.
\end{equation*}
We then check that $K(u)=\sum_{0\le s\le s-1}a_su^{s+\kappa}$ satisfies
\begin{equation*}
  u^{h+1}K^{(h+1)}(u)
  =\frac{(\kappa+h)!}{(\kappa-1)!}K(u)
+(h+1)!P(u)u^{\kappa-1}.
  \end{equation*}
Note that we could have added any monomial $a_hu^{h+\kappa}$ to $K(u)$.
\subsection*{From the approximate differential equation to the exact one}
We define $W(u)=H_{h+1}(u)u^{-1}-K(u)$. This function satisfies
\begin{equation*}
  \label{eq:9}
  u^{h+1}W^{(h+1)}(u)
  =\frac{(\kappa+h)!}{(\kappa-1)!}W(u)
  +\Ocal(u^{\kappa-1}(\log u)^{\frac{h(h-1)}{2}}).
\end{equation*}
We are in good conditions to use Lemma~\ref{popa}. At the beginning, we should consider the roots $\lambda_1=\kappa+h,\cdots, \lambda_r$ of the equation
\begin{equation*}
    \lambda(\lambda-1)\cdots(\lambda-h)
    =\kappa(\kappa+1)\cdots(\kappa+h)
\end{equation*}
that are such that $\lambda_i>\kappa-1$.  Set
$\varphi(u)=Cu^{\kappa-1}(\log 2u)^{\frac{h(h-1)}{2}}$ for a large
enough constant~$C$, so that
\begin{equation*}
    \biggl|u^{h+1}W^{(h+1)}(u)
  -\frac{(\kappa+h)!}{(\kappa-1)!}W(u)
  \biggr|\le \varphi(u).
\end{equation*}
We find that
\begin{equation*}
    \Phi^*_{\lambda_i,c_i}(\varphi)(u)
    =Cu^{\Re \lambda_i}\biggl|
    \int_{c_i}^u t^{(\kappa-1-\Re \lambda_i) u}\log(2t)^{\frac{h(h-1)}{2}}dt
    \biggr|.
\end{equation*}
When $\kappa-\Re\lambda_i>0$, we select $c_i=1$ and get that
$\Phi^*_{\lambda_i,c_i}(\varphi)(u)\ll u^{\kappa-1}\log(2u)$. When
$\kappa-\Re\lambda_i<0$, we select $c_i=\infty$ and get a same
result. There remains the case $\kappa=\Re \lambda_i$ where we select
$c_i=1$ and get a further power of $\log u$.
By Lemma~\ref{popa}, there exist parameters $C_1,\cdots,C_r$ such that
\begin{equation*}
  \label{eq:3}
  \Bigl|W(u)-\sum_{1\le s\le r}C_s u^{\lambda_s}\Bigr|\le
  u^{\kappa-1}(\log 2u)^{\frac{h(h+1)}{2}}.
\end{equation*}
At this level, we still have not proved that the relevant roots
$\lambda_s$ that have a non-zero coefficient $C_s$ are of the form
$\kappa+h-\ell$.

\subsection*{From $W$ to $H_{h+1}$}
The determination of $W$ via~\eqref{eq:3} goes to $H_{h+1}^{(h+1)}$
by~\eqref{eq:9} and the definition
$W(u)=H_{h+1}(u)u^{-1}-K(u)$. We thus obtain that
\begin{equation*}
  \label{eq:4}
  \sum_{n\le X}g(n)(\log n)^{h+1}
  =\sum_{i}C_i(\log X)^{\theta_i}
  +\Ocal\Bigl((\log X)^{\kappa}(\log\log (3X))^{\frac{h(h+1)}{2}}\Bigr)
\end{equation*}
where the sequence $(\theta_i)$ is the union of the one of $\lambda_s$
and of $\kappa+h,\kappa+h-1,\ldots,\kappa$, coming
from~$K(u)$. By our functional equation~\eqref{step2}, we have a
similar development when we replace $(\log n)^{h+1}$ by $(\log X/n)^{h+1}$.

\section{Ruling out the parasiting solutions}
When $h=1$, the two roots are $\kappa+1$ and
$-\kappa$. Lemma~\ref{popa} then implies that we can find $a$ and $b$
such that
\begin{equation*}
    |W(u)-au^{\kappa+1}-bu^{-\kappa}|
    \ll u^{\kappa-1}\log(2u).
\end{equation*}
This reduces to $|W(u)-au^{\kappa+h}| \ll u^{\kappa-1}\log(2u)$
when $\kappa\ge 1/2$. But what happens when $\kappa<1/2$~? 
\subsection*{A stability remark}
Assume we have a non-negative multiplicative function $f$ that
satisfies the assumptions of our Theorem~\ref{thmmain}. Assume further
we have distinct exponents $\kappa_0=\kappa,\kappa_1,\ldots,\kappa_r\ge\kappa$ such that
\begin{equation*}
    \sum_{n\le X}\frac{f(n)}{n}
    \biggl(\log\frac{X}{n}\biggr)^{h+1}
    =\sum_{0\le s\le r}C_s(\log X)^{h+1+\kappa_s}
    +\Ocal((\log X)^{\kappa}(\log\log (3X))^C)
\end{equation*}
for some non-zero constants $C_0,\ldots,C_r$ and $C\ge0$.  Select a positive
integer $K$ and consider the function $\tau_K$ that counts the number of
$K$-tuples of divisors, so that $\tau_2$ is the usual divisor
function. Next we consider the multiplicative function $f\star \tau_K$
that equally satisfies the assumptions of Theorem~\ref{thmmain},
though with $\kappa+K$ instead of $\kappa$. By
the Dirichlet Hyperbola Formula, we find that
\begin{multline*}
    \sum_{n\le X}\frac{(f\star \tau_K)(n)}{n}
    \biggl(\log\frac{X}{n}\biggr)^{h+1}
    =\sum_{0\le s\le r}\mkern-20mu\sum_{\substack{\ell\ge0,\\
        h+1+K+\kappa_s-\ell>\kappa-1}}\mkern-15mu
    C'_{s,\ell}(\log X)^{h+1+K+\kappa_s-\ell}
    \\+\Ocal((\log X)^{K+\kappa}(\log\log (3X))^C)
\end{multline*}
for some constants $C'_0,\ldots,C'_r$. This tells us that the set of
exponents for $f\star\tau_K$ is $\kappa_0+K,\ldots, \kappa_r+K$. Let
$\kappa_s$ denotes the largest, if it exists, of the $\kappa_i$'s that
is not of the form $\kappa+h$ minus some integer. Then the coefficient
$C'_{s,0}$ comes from the main term of
\begin{equation*}
  C_s\sum_{n\le X}\frac{\tau_K(n)}{n}\biggl(\log\frac{X}{n}\biggr)^{h+1+K+\kappa_s}
\end{equation*}
and is thus a non-zero multiple of $C_s$.

\subsection*{General case}
In general the discussion of
previous subsection applies: we only need to consider the roots of $\lambda$ of
\begin{equation*}
    R_h(\lambda,\kappa)
    =\lambda(\lambda-1)\cdots(\lambda-h)
    -\kappa(\kappa+1)\cdots(\kappa+h)
\end{equation*}
that are such that $R_h(\lambda+K,\kappa+K)=0$ when $K$ is a positive
integer. This leads to a polynomial in $K$ of degree~$h+1$ that
vanishes at these points $(\lambda,\kappa)$.  The coefficient of $K^h$
is
\begin{equation*}
    (h+1)\lambda-(1+2+\ldots+h-1)
    -(h+1)(\kappa+h)+(1+2+\ldots+h-1)
\end{equation*}
and since it vanishes, we must have $\lambda=\kappa+h$. In short: only
integer translates of $\kappa$ may appear, and this concludes the
proof of Theorem~\ref{thmmain}.
\section{Technical remarks}
\label{Technical}
The Levin-\Fainleib's beginning, namely the link between
$\sum_{n\le x}g(n)\log n$ and $\sum_{n\le x}g(n)$ where $g(n)=f(n)/n$,
has had many application, so it is worse providing a sketch of the
present method. When $h=1$ and $f$ is restricted to
square-free integers, our method relies on the identity (as noticed
immediately after Lemma~\ref{ident}):
\begin{multline*}
  \sum_{n\le X}\frac{f(n)}{n}(\log n)^2=
  \sum_{m\le X}\frac{f(m)}{m}
  \\\biggl(\sum_{p_1p_2\le X/m}\frac{f(p_1)f(p_2)(\log p_1)(\log p_2)}{p_1p_2}
  +\sum_{p\le X/m}\frac{f(p)(\log p)^2}{p}\biggr)+\text{error}.
\end{multline*}
A similar equation could be reached by noticing that, by the Selberg
Formula $\log^2=\1\star(\Lambda\log+\Lambda\star\Lambda)$, we have
\begin{multline*}
  \sum_{n\le X}\frac{f(n)}{n}(\log n)^2=
  \sum_{m\le X}\frac{f(m)}{m}
  \\\biggl(\sum_{\substack{p_1p_2\le X/m,\\ (p_1p_2,m)=1}}\frac{f(p_1p_2)(\log p_1)(\log p_2)}{p_1p_2}
  +\sum_{\substack{p\le X/m,\\ (p,m)=1}}\frac{f(p)(\log p)^2}{p}\biggr).
\end{multline*}
Our usage of $\Lambda_f$ thus avoids the coprimality conditions that
soon become a true combinatorial hurdles. Then by Lemma~\ref{prep} (or
Lemma~\ref{lifeagain}), we approximate the sum of the two sums over
primes above by $\kappa(\kappa+1)(\log Y)^2 + c(\log Y)+\Ocal(\log\log
3Y)$ and we notice that
\begin{equation*}
  (\log n)^2=\Bigl(\log X-\log\frac{X}{n}\Bigr)^2
  =(\log X)^2-2(\log X)\log\frac{X}{n}+\Bigl(\log\frac{X}{n}\Bigr)^2.
\end{equation*}
This gives us
\begin{multline*}
  (\log X)^2G_0(X)-2(\log X)G_1(X)+G_2(X)
  =\kappa(\kappa+1) G_2(X)
  \\+\Ocal(G_0(X)\log\log 3X).
\end{multline*}
We then convert this in an approximate differential equation in $H_2$
of Euler's type, i.e. it can be reduced to an approximate linear
differential equation, for which one can prove deformation results.

\bibliographystyle{abbrv}
\bibliography{bibliography}


\end{document}